\documentclass[10pt]{article}
\font\smallit=cmti10

\usepackage{amssymb,latexsym,amsmath,epsfig,amsthm} 

\makeatletter

\renewcommand\section{\@startsection {section}{1}{\z@}
{-30pt \@plus -1ex \@minus -.2ex}
{2.3ex \@plus.2ex}
{\normalfont\normalsize\bfseries\boldmath}}

\renewcommand\subsection{\@startsection{subsection}{2}{\z@}
{-3.25ex\@plus -1ex \@minus -.2ex}
{1.5ex \@plus .2ex}
{\normalfont\normalsize\bfseries\boldmath}}

\renewcommand{\@seccntformat}[1]{\csname the#1\endcsname. }

\makeatother

\newtheorem{theorem}{Theorem}
\newtheorem{lemma}{Lemma}
\newtheorem{proposition}{Proposition}

\theoremstyle{definition}

\newtheorem{remark}{Remark}

\usepackage[all]{xy}
\usepackage{a4wide}
\usepackage{amssymb}
\usepackage{caption}
\usepackage{amsthm}
\usepackage{amsmath}
\usepackage{amscd,enumitem}
\usepackage{verbatim}
\usepackage{float}
\usepackage{color}
\usepackage{wasysym}
\usepackage[all]{xy}
\usepackage{hyperref}
\usepackage{cleveref}
\theoremstyle{plain}

\theoremstyle{definition}
\newtheorem{example}[theorem]{Example}

\begin{document}

\begin{center}
\uppercase{\bf Combinatorial formulas for arithmetic density}
\vskip 20pt
{\bf Robert Schneider}\\
{\smallit Department of Mathematics, University of Georgia, Athens, Georgia, USA}\\
{\tt robertsc@mtu.edu}\\ 
\vskip 10pt
{\bf Andrew V. Sills}\\
{\smallit Department of Mathematical Sciences, Georgia Southern University
, Statesboro and Savannah, Georgia, USA}\\
{\tt ASills@GeorgiaSouthern.edu}\\ 
\end{center}
\vskip 30pt

\centerline{\bf Abstract}
\noindent
Let $d_S$ denote the arithmetic density of a subset $S \subseteq \mathbb N$. 
 We derive a power series in $q\in \mathbb C$, 
 $|q|<1$, with co\"efficients related to integer partitions and integer compositions,
 that yields $1/d_S$ in the limit as $q\to 1$ radially. 

%

\section{Introduction and statement of results}\label{Sect1}

In recent works \cite{OSW, OSW2}, Ono, Wagner, and the first author use methods from partition theory to prove 
$q$-series formulas for the arithmetic density $d_S$ of a subset 
$S$ of natural numbers $\mathbb N$: 
\begin{equation}
d_{S}:=\lim_{N \to \infty} \frac{\# \{ n \in S : n \leq N \}}{N}.\end{equation}
In this paper, we prove a combinatorial limiting formula for the reciprocal value $1/d_S$, using ideas related to both partitions and integer compositions.

Let $\mathcal P$ denote the set of {\it integer partitions}, unordered finite sums of natural numbers (see e.g. \cite{And}), including the empty partition $\emptyset\in \mathcal P$. Let $S\subseteq \mathbb N$, and let $\mathcal P_S$ denote the set of partitions whose parts lie in the subset $S$. 
For $\lambda \in \mathcal P$, let $|\lambda|\geq 0$ denote the {\it size} of the partition (sum of parts), let $\ell(\lambda)\geq 0$ denote the {\it length} (number of parts), and let $m_i=m_i(\lambda)\geq 0$ be the {\it multiplicity} (frequency) of $i\in \mathbb N$ as a part of partition $\lambda$.

Define the sum $C_S(n)$ over partitions $\lambda\in\mathcal P_S$ with size $|\lambda|\leq n$, noting $m_i(\lambda)=0$ if $i\not\in S$, by
\begin{equation}\label{sum}
C_S(n):=\sum_{\substack{\lambda \in \mathcal P_S\\ 0\leq |\lambda|\leq n}}\frac{(-1)^{\ell(\lambda)}\ell(\lambda)!}{m_{1}!m_{2}! m_{3}!\cdots m_{n}!}.\end{equation}
Sums over partitions have a history dating
back to work of MacMahon~\cite[p. 61ff.]{MacMahon} and Fine~\cite[\S22]{Fine}. They are important in modern number theory; the $q$-bracket of Bloch--Okounkov \cite{B-O, Zagier_bracket} is an operator from statistical physics that induces modularity in partition-theoretic $q$-series. 

Equation \eqref{sum} has a natural combinatorial interpretation in terms of {\it integer compositions}, which are ordered finite sums of natural numbers (see e.g. 
\cite[Section IV]{compositions}). 
Let $\mathcal C$ denote the set of all integer compositions. We will extend the partition terms and notations defined above to compositions, with the same meanings. 
Let $\mathcal C_S$ denote the compositions whose parts lie in $S\subseteq \mathbb N$. For each partition $\lambda$ in \eqref{sum}, $\ell(\lambda)!/m_1! m_2 ! m_3! \cdots$ counts the  multiset permutations of the parts of $\lambda$, i.e., all compositions $\gamma\in\mathcal C_S$ having  the same parts as $\lambda$.

\begin{proposition} $C_S(n)$ counts the number of {compositions} $\gamma$ in $\mathcal C_S$ of even length, minus those of odd length, having sizes $0\leq |\gamma|\leq n$:
$$C_S(n)=\sum_{\substack{\gamma\in \mathcal C_S\\ 0\leq |\gamma|\leq n}}(-1)^{\ell(\gamma)}.$$
\end{proposition}

Let $F_S(q):=\sum_{n\geq 0} C_S(n)q^n$ with domain of convergence depending on $S\subseteq \mathbb N$; and define the auxiliary series $f_S(q):= 1+\sum_{n\in S}q^n, |q|<1$. Our focus is on the behavior of 
$F_S(q)$ as $q\to 1$.

\begin{theorem}\label{thm1}  Let $S\subseteq \mathbb N$ such that $d_S> 0$, and such that $f_S(q)=1+\sum_{n\in S}q^n$ is analytic and has no zeros on $\{q\in \mathbb C : |q|<1\}$. Then as $q\to 1$ radially, we have
$$\lim_{q\to 1}F_S(q)\  =\  \lim_{q\to 1}\  \sum_{n=0}^{\infty}C_S(n)q^n\  =\  \frac{1}{d_S}.$$
\end{theorem}

Theorem \ref{thm1} has something of an analytic converse. 

\begin{theorem}\label{thm2}  Let $S\subseteq \mathbb N$ such that $F_S(q)=\sum_{n=0}^{\infty}C_S(n)q^n$ is analytic and has no zeros on $\{q\in \mathbb C : |q|<1\}$, and such that 
$\lim_{q\to 1}F_S(q)=L$ exists as $q\to 1$ radially. If the limit $L$ is infinite, then $d_S=0$. If $L\geq 1$ is finite, then $d_S=1/L$. The case $L<1$ cannot occur.
\end{theorem}

We postpone proofs until Section \ref{sect2}. Here is an example using Theorem \ref{thm1}. 
\begin{example}\label{cor1} For $t, r\in\mathbb N$, with $r\leq t$, let $S_{r, t}$ denote the set of positive integers congruent to $r$ modulo $t$. Then  as $q\to 1$ radially, we have
\begin{equation*}\lim_{q\to 1}\  \sum_{n=0}^{\infty}C_{S_{r,t}}(n)q^n\  = t.\end{equation*}
We confirm that $d_{S_{r, t}}=1/t>0$. Moreover, since $f_{S_{r, t}}(q)=1+\sum_{n\in S_{r, t}}q^n=1+\sum_{n\geq 0}q^{r+nt}=(1+q^r-q^t)/(1-q^t)$ is analytic on $|q|<1$ with no zeros in the unit disk, then the series satisfies the analytic conditions in Theorem \ref{thm1}. The right-hand side gives $1/d_{S_{r, t}}=t$.  Note that this limit can be seen directly, as 
$\lim_{q\to 1} (1-q) (1 + q^r - q^t)/ (1-q^t) = 1/t$
by L'H\^{o}pital's rule. 
\end{example}

\section{Proofs of Theorem \ref{thm1} and Theorem \ref{thm2}}\label{sect2}

We prove the theorems using the multinomial theorem, geometric series and the Cauchy product formula for power series, together with a theorem of Frobenius \cite{Frobenius}.  Our central lemma expresses the co\"efficients of the reciprocal of
a power series as a sum over partitions.  

\begin{lemma}\label{lemma1}
For $a_i\in\mathbb C, a_0\neq 0$, let $f(q):=\sum_{n\geq 0}a_n q^n$ be analytic on $\{q\in \mathbb C : |q|<1\}$. Then on the domain of analyticity of $\phi(q):=1/f(q)$ we have 
\[ \phi(q) = \frac{1}{f(q)}\  =\  \sum_{n=0}^{\infty}c_n q^n ,\  \  \mbox{where} \ 
c_n\  =\ 
{\sum_{\substack{\lambda\in\mathcal{P}\\|\lambda| = n }}}
\frac{(-1)^{\ell(\lambda)}\ell(\lambda)!\  a_1^{m_1}a_2^{m_2}a_3^{m_3}\cdots a_n^{m_n}}{a_0^{\ell(\lambda)+1}\  m_1!m_2! m_3!\cdots m_n!}. 
\]
\end{lemma}

\begin{proof}[Proof of Lemma \ref{lemma1}]
The result is equivalent to \cite[Thm. 3.1]{Salem}, 
which is proved using the Maclaurin expansion $\phi(q)=1/f(q)=\sum_{n\geq 0}\phi^{(n)}(0) q^n/n!=\sum_{n\geq 0}c_n q^n$. 
For completeness, we give a self-contained proof of the identity for $c_n$. Begin with the {\it multinomial theorem} (see e.g. \cite{multinomial}), written as a sum over partitions $\lambda$ with {\it largest part} $\operatorname{lg}(\lambda)\leq k,\  \operatorname{lg}(\emptyset):=0,$ and length $\ell(\lambda)=r$: 
\begin{equation}\label{Dmultinomial}
(x_1+x_2+x_3+\dots+x_k)^r=r! \sum_{\substack{0\leq \operatorname{lg}(\lambda)\leq k\\\ell(\lambda)=r}} \frac{x_1^{m_1}x_2^{m_2}x_3^{m_3}\cdots x_k^{m_k}}{m_1!\  m_2!\  m_3!\  \cdots \  m_k!}.
\end{equation}
Make the substitution $x_i = \frac{a_i}{a_0} q^i$ with $a_i$ as in Lemma \ref{lemma1}. We now let $k$ tend to infinity; if $g(q):=\frac{a_1}{a_0} q+\frac{a_2}{a_0} q^2+\frac{a_3}{a_0} q^3+\dots=a_0^{-1}f(q)-1$ converges absolutely on $|q|<1$, then \eqref{Dmultinomial} becomes \begin{equation}\label{Dmultinomial2}
\left(g(q)\right)^r= \frac{r!}{a_0^{r}} \sum_{\ell(\lambda)=r} q^{|\lambda|}\frac{a_1^{m_1}a_2^{m_2}a_3^{m_3}\cdots}{m_1!\  m_2!\  m_3!\  \cdots }.
\end{equation}
For $q\in\mathbb C$ such that $|g(q)|<1$, multiplying both sides of \eqref{Dmultinomial2} by $(-1)^r$ and summing over $r\geq 0$ gives an infinite geometric series on the left, and a sum over all partitions on the right:
\begin{equation}\label{Dmultinomial3}
\frac{1}{1+g(q)}\  =\  \frac{a_0}{f(q)}\  =\  \sum_{\lambda\in\mathcal P} q^{|\lambda|}\frac{(-1)^{\ell(\lambda)}\ell(\lambda)!\  a_1^{m_1}a_2^{m_2}a_3^{m_3}\cdots  }{a_0^{\ell(\lambda)}\  m_1!\  m_2!\  m_3!\  \cdots }.
\end{equation}
Dividing through by $a_0$, then collecting co\"{e}fficients of $q^n$ on the right-hand side to write $1/f(q)=\sum_{n\geq 0}c_n q^n$, proves the identity for $c_n$. 
While the geometric series representation of $a_0/f(q)$ is only valid when $|g(q)|<1$,  the formula for $c_n=\phi^{(n)}(0) /n!$ holds on the domain of analyticity of $\phi(q)=1/f(q)$ by uniqueness of the Maclaurin series representation of the analytic function, noting $f(0)\neq 0$ so $1/f(q)$ can be expressed as a power series centered at $q=0$.
\end{proof}

\begin{lemma}\label{lemma2}
Let $a_i\in\mathbb C, a_0\neq 0,$ be such that $f(q)=\sum_{n\geq 0}a_n q^n$ is analytic and has no zeros on $\{q\in \mathbb C : |q|<1\}$. Define $c_n$ as in Lemma \ref{lemma1}. Define $A(n):=\sum_{i=0}^n a_i,\  C(n):=\sum_{i=0}^n c_i$, and let $A:=\lim_{n\to \infty} {A(n)}/{n}$  if the limit exists. Then if $A\neq 0$, as $q\to 1$ radially we have that 
$$\lim_{q\to 1}\  \sum_{n=0}^{\infty} C(n)q^n\  =\  \frac{1}{A}.$$
\end{lemma}

\begin{proof}[Proof of Lemma \ref{lemma2}]
This lemma can be obtained as a special case of the first asymptotic formula in \cite{Hardy}; for 
completeness, we prove it directly. For the claimed limit to exist as $q\to 1$, we need the Maclaurin 
series for $\phi(q)=1/f(q)$ convergent on the unit disk $|q|<1$. Take $\varepsilon>0$ and suppose 
that $f(q_0)=0$ for some $q_0$ with $|q_0|<1-\varepsilon$. Then $\phi(q_0)=1/f(q_0)$ represents a 
pole, so the series $\phi(q)$ has radius of convergence $\leq 1-\varepsilon$ and the limit as $q\to 1$ 
does not exist. Hence the conditions on the domain and zeros are necessary. That this limit is equal 
to $1/A$ follows from \cite{Frobenius}, which proves when $q\to 1$ radially from within the unit 
disk\footnote{In fact, this limiting result holds if $q\to 1$ through any path in a {Stolz sector} of the unit 
disk (see e.g. \cite{Stolz}), a region with vertex at $q=1$ such that $\frac{|1-q|}{1-|q|}\leq M$ for some 
$M> 0$.} that  
\begin{equation}\label{Frob2} \lim_{q\to 1}(1-q)f(q)\  =\  A.\end{equation} 
Noting that $1/(1-q),$ $\phi(q)$, and thus $(1-q)^{-1}\phi(q)$ are analytic on $|q|<1$, then Lemma \ref{lemma1} gives 
\begin{equation}\label{Cauchy} \frac{1}{A}\  =\  \lim_{q\to 1}\frac{1}{1-q}\sum_{n\geq 0}c_nq^n\  =\  \lim_{q\to 1}\left(\sum_{n\geq 0}q^n\right)\left(\sum_{n\geq 0}c_nq^n\right)
\  =\  \lim_{q\to 1}\sum_{n\geq0} C(n)q^n,\end{equation}
with $C(n)=\sum_{0\leq i \leq n}c_i=\sum_{0\leq i \leq n}1\cdot c_i$ due to the Cauchy product formula for power series. 
\end{proof}

\begin{proof}[Proof of Theorem \ref{thm1}]
Set $a_0=1$ and for $n\geq 1$, let $a_n$ be the indicator function of $S\subseteq \mathbb N$, i.e., $a_n=1$ if $n\in S$, and $a_n=0$ if $n\not\in S$.  Since $1+\sum_{n\in S}q^n=\sum_{n\geq 0}a_n q^n$ is analytic on $|q|<1$ by comparison with geometric series, then with the stipulation it has no zeros, $f_S(q)$ satisfies the analytic conditions of Lemmas \ref{lemma1} and \ref{lemma2}.  Thus, $a_1^{m_1}a_2^{m_2}\cdots a_r^{m_r}=1$ if $\lambda\in\mathcal P_S$ and $=0$ otherwise in the expression for $c_n$ in Lemma \ref{lemma1}, which yields $C(n)=C_S(n)$ in Lemma \ref{lemma2}. Observing that $A= \lim_{n\to \infty} A(n)/n=\lim_{n\to \infty} {\# \{ i \in S : i \leq n \}}/{n} = d_S$ in Lemma \ref{lemma2} gives the theorem. \end{proof}

\begin{proof}[Proof of  Theorem \ref{thm2}]
Since $F_S(q)$ is analytic with no zeros for $|q|<1$, then $\Phi_S(q):=1/F_S(q)$ is also analytic with no zeros inside the unit disk, and has a unique power series expansion $\Phi_S(q)=\sum_{n\geq 0}d_n q^n$ around the origin. In Theorem \ref{thm1}, $1/F_S(q)$ has the power series expansion $(1-q)\left(1+\sum_{n\in S}q^n\right)=(1-q)f_S(q)$, analytic for $|q|<1$. Then by uniqueness of the power series expansion of $1/F_S(q)$ on its domain of analyticity, $\Phi_S(q)=(1-q)f_S(q)$ on the unit disk. 
If $L$ is infinite, then $\lim_{q\to 1}\Phi_S(q):=\lim_{q\to 1}1/F_S(q)=0$ which is equal to $d_S$ by \eqref{Frob2}. 
If $L\geq1$ is finite, then $1/L=\lim_{q\to 1}1/F_S(q)=\lim_{q\to 1}\Phi_S(q)=d_S$, also by \eqref{Frob2}. If the case $L<1$ were to occur under these hypotheses, it would mean $1/F_S(q)=(1-q)f_S(q)\to 1/L>1$ as $q\to 0$. But with all co\"{e}fficients of $\sum_{n\geq 0}a_n q^n = 1+\sum_{n\in S}q^n$ being $0$ or $1$, then by  \eqref{Frob2}, $\lim_{q\to 1}(1-q)f_S(q) = \lim_{N\to \infty}\frac{1}{N}\sum_{n=0}^{N-1} a_n   \leq  1,$  
if the limit exists. Thus $L<1$ cannot occur.     \end{proof}

\section{Further remarks}\label{sect3}

Let $c(n) = 2^{n-1}$ denote the number of compositions of size $n\geq 1$ ~\cite[p. 151]{compositions}, 
with $c(0):=1$, and let $c_S(n)$ denote the number of size-$n$ compositions having all parts from  
$S\subseteq \mathbb N$. Considering the results in Section \ref{Sect1}, one wonders about the 
``non-alternating'' variant of  \eqref{sum}: 
\begin{equation} C_S^+(n):=\sum_{\substack{\lambda \in \mathcal P_S\\0\leq |\lambda|\leq n}}
\frac{\ell(\lambda)!}{m_{1}!m_{2}! m_{3}!\cdots m_{n}!}=\sum_{0\leq  j \leq n}c_S(j),\end{equation} 
the number of compositions $\gamma \in \mathcal C_S$ having sizes $0\leq |\gamma|\leq n$. This is 
a fairly natural statistic, e.g. for $S=\mathbb N$ one has $C_{\mathbb N}^+(n)=\sum_{0\leq j \leq n}
c(j)=1 + (1+2+2^2+2^3+\dots+2^{n-1})=2^n$. 

What are the analytic properties of the power series  $F_S^+(q):=\sum_{n\geq 0}C_S^+(n) q^n,\  |q|
<1$? 
Let $g_S(q):=f_S(q)-1$ with $f_S(q)$ analytic on $|q|<1$. Similar multinomial, geometric series, 
Cauchy product and analytic arguments as above, applied to $\sum_{r\geq 0} 
\left(g_S(q)\right)^r, |g_S(q)|<1,$ prove the generating function formula   $\sum_{n\geq 0}
c_S(n)q^n=\sum_{\gamma\in\mathcal C_S}q^{|\gamma|}=(2-f_S(q))^{-1}=(1-g_S(q))^{-1}$  (see e.g.  
\cite[Thm. 1.1]{gf}), 
for $q\in\mathbb C$ such that $|f_S(q)|<2$. Then by \eqref{Dmultinomial3}, together with the Cauchy 
product for power series as used  in \eqref{Cauchy} above, we deduce the identity 
\begin{equation}\label{equality2}
F_S^+(q)\  =\  \sum_{n=0}^{\infty}C_S^+(n)q^n\  =\  \frac{1}{(1-q)\left(2-f_S(q)\right)}.
\end{equation} 

However, limiting formulas analogous to Theorem \ref{thm1} do not result in this case, as  $F_S^+(q)
$ is not analytic on the unit disk. To see this, note that if $S$ is a finite nonempty subset, then $|
f_S(q)|\leq 1+\#S$ when $|q|\leq 1$ since there are $\#S$ terms of the form $q^n$, with $f_S(1)=1+
\#S$ exactly. If $S$ is an infinite subset of $\mathbb N$, then $|f_S(q)|\to \infty$ as $q\to 1$. Thus $|
f_S(q)|<2$ for all $|q|<1$ if and only if the subset $S$ has one element. For a subset $S$ with two or 
more elements, $F_S^+(q)$ converges on a disk strictly smaller than the unit disk, and $\lim_{q\to 1}
F_S^+(q)$ does not exist.

\begin{remark} It is possible to find 
composition-theoretic limiting formulas in a smaller disk, e.g.   

\begin{equation} \label{half}
\lim_{q\to 1/2}\  \frac{1-2q}{1-q}\sum_{n\in S}c(n)q^n\  =\  \lim_{q\to 1/2}\  \frac{\sum_{n\in S}c(n)q^n}{\sum_{n\geq 0}c(n)q^n}\  
=\   {d_S}.\end{equation} 

This can be deduced from \eqref{Frob2}: 
\begin{equation*}  d_S=\lim_{q\to 1}\  (1-q){\sum_{n\in S} q^n}\  =\  \lim_{q\to 1/2}\  (1-2q){\sum_{n\in S} (2q)^n}\  =\  \lim_{q\to 1/2}\ \frac{1-2q}{1-q}{\sum_{n\in S} 2^{n-1}q^n},\end{equation*} 
noting that $\sum_{n\geq 0}c(n)q^n=1+\sum_{n\geq 1}2^{n-1}q^n=(1-q)/(1-2q)$.\end{remark}

\section*{Acknowledgments}
The authors are thankful to Maurice D. Hendon for advice on convergence of complex power series 
that strengthened our proofs, and for suggesting the special case of Example \ref{cor1}; to George E. 
Andrews for noting a useful correction in an earlier draft; and to the anonymous referee for carefully 
reviewing our work.

\end{document}